\documentclass[12pt]{article}
\usepackage{amsmath,amsfonts,amssymb,mathtools}
\usepackage{amsthm} 
\usepackage{amsmath}
\usepackage[hmargin=2.5cm,vmargin=3cm,centering]{geometry}
\usepackage{graphicx}
\usepackage{hyperref}
\hypersetup{
    colorlinks,
    citecolor=black,
    filecolor=magenta,
    linkcolor=black,
    urlcolor=black,
    }
\usepackage[titletoc]{appendix}
\usepackage{color}
 \definecolor{red}{rgb}{1,0,0}

\newcommand{\xddots}{%
  \raise 4pt \hbox {.}
  \mkern 6mu
  \raise 1pt \hbox {.}
  \mkern 6mu
  \raise -2pt \hbox {.}
}

 \theoremstyle{break}
\newtheorem{thm}{Theorem}[section] 
\newtheorem{RQ}[thm]{Remark}

\newtheorem{prop}[thm] {Proposition}
\newtheorem{cor}[thm] {Corollary}

\numberwithin{equation}{section}
\newtheorem{lem}[thm]{Lemma}

\newtheorem{assumption}{Assumption}
\usepackage{xcolor}
\usepackage{enumitem}

\newenvironment{outerdesc}[1][black]
  {\begin{description}[font=\normalfont\color{#1}]}
  {\end{description}}

\usepackage{float}

\title{\large\bf{Kramers-Fokker-Planck
 operators with homogeneous potentials}} 

\author{
   Mona Ben Said\\
Laboratoire Analyse, G{\'e}om{\'e}trie et Applications\\
   Universit{\'e} Paris 13\\
   99 Avenue Jean Baptiste Cl{\'e}ment \\93430 Villetaneuse, France\\
  bensaid@univ-paris13.fr 
   }

\begin{document}

\maketitle
\begin{abstract}
In this article we establish a global subelliptic estimate for Kramers-Fokker-Planck operators with homogeneous potentials $V(q)$ under some conditions, involving in particular the control of the eigenvalues of the Hessian matrix of the potential. Namely, this work presents a different approach from the one in \cite{Ben}, in which the case $V(q_1,q_2)=-q_1^2(q_1^2+q_2^2)^n$ was already treated only for $n=1.$ With this article, after the former one dealing with non homogeneous polynomial potentials, we conclude the analysis of all the examples of degenerate ellipticity at infinty presented in the framework of Witten Laplacian by Helffer and Nier in \cite{HeNi}. Like in \cite{Ben}, our subelliptic lower bounds are the optimal ones up to some logarithmic correction.
\end{abstract}

\noindent\textbf{MSC-2010:} 35Q84, 35H20, 35P05, 47A10, 14P10
\tableofcontents
\section{Introduction and main results}
In this work we study the Kramers-Fokker-Planck operator
\begin{align}
K_V=p.\partial_q-\partial_qV(q).\partial_p+\frac{1}{2}(-\Delta_p+p^2)~,\;\;\;\;\;(q,p)\in
  \mathbb{R}^{2d}
\,,
\label{a.3eq1}
\end{align}
where $q$ denotes the space variable, $p$ denotes the velocity
variable and the potential $V(q)$
is a real-valued function defined in the
whole space $\mathbb{R}^d_q.$ 

Setting 
\[
    O_p=\frac{1}{2}(D^2_p+p^2)
\;,\quad\quad \text{and}\quad\quad
X_V=p.\partial_q-\partial_qV(q).\partial_p~,
\] 
the Kramers-Fokker-Planck operator $K_V$ defined in (\ref{a.3eq1}) reads
$K_V=X_V+O_p.$ \\
We firstly list some notations used throughout the paper. We denote for an arbitrary function $V(q)$ in $\mathcal{C}^{\infty}(\mathbb{R}^d)$
\[
    \begin{aligned}
    \mathrm{Tr}_{+,V}(q) & = \sum\limits_{\substack{\nu\in \mathrm{Spec}(\mathrm{Hess}\; V)\\ \nu>0}} \nu(q)\,,
    \\ \mathrm{Tr}_{-,V}(q) &=-\sum\limits_{\substack{\nu\in \mathrm{Spec}(\mathrm{Hess}\; V)\\ \nu\le 0}}\nu(q)\;.
    \end{aligned}
\]
In particular for a polynomial $V$ of degree less than 3, $\mathrm{Tr}_{+,V}$ and $\mathrm{Tr}_{-,V}$ are two constants. In this case we define the constants $A_V$ and $B_V$ by
\begin{align*}
   A_V& = \max \{(1+\mathrm{Tr}_{+,V})^{2/3}, 1+\mathrm{Tr}_{-,V}\}\;,\\
 B_V &= \max\{\min\limits_{q\in\mathbb{R}^d}\left|\nabla\;V(q)\right|^{4/3}, \frac{1+\mathrm{Tr}_{-,V}}{\log(2+\mathrm{Tr}_{-,V})^2}\}\;. \end{align*} 
This work is principally based on the publication by Ben Said,
Nier, and Viola \cite{BNV}, which concerns the study of
Kramers-Fokker-Planck operators with polynomials of degree less than
three. In \cite{BNV} we proved the existence of a  constant $c>0$,
independent of $V$, such that the following global subelliptic estimate with remainder
\begin{align}
\|K_Vu\|^2_{L^2(\mathbb{R}^{2d})}+A_V\|u\|^2_{L^2(\mathbb{R}^{2d})}\ge
{c} \Big(\|O_pu\|^2_{L^2(\mathbb{R}^{2d})}&+\|X_Vu\|^2_{L^2(\mathbb{R}^{2d})}\nonumber\\&+\|\langle\partial_q V(q)\rangle^{2/3}u\|^2_{L^2(\mathbb{R}^{2d})}+\|\langle D_q\rangle^{2/3}u\|_{L^2(\mathbb{R}^{2d})}\Big)\label{eq44}
\end{align}
holds for all $u\in \mathcal{C}_0^{\infty}(\mathbb{R}^{2d}).$
Furthermore, supposing
$\mathrm{Tr}_{-,V}+\min\limits_{q\in\mathbb{R}^d}\left|\nabla\;V(q)\right|\not=0$,
there exists a constant $c>0$, independent of $V$, such that 
\begin{align}
\|K_Vu\|^2_{L^2(\mathbb{R}^{2d})}\ge c\,B_V\|u\|^2_{L^2(\mathbb{R}^{2d})}~,\label{1.5mm}
\end{align}
is valid for all $u\in \mathcal{C}_0^{\infty}(\mathbb{R}^{2d}).$
As a consequence collecting (\ref{1.5mm}) and (\ref{eq44}) together,
there is a constant $c>0$, independent of $V$, so that the global subelliptic estimates without remainder
\begin{align}
\|K_Vu\|^2_{L^2(\mathbb{R}^{2d})}\ge \frac{c}{1+\frac{A_V}{B_V}}\Big(\|O_pu\|^2_{L^2(\mathbb{R}^{2d})}&+\|X_Vu\|^2_{L^2(\mathbb{R}^{2d})}\nonumber\\&+\|\langle\partial_q V(q)\rangle^{2/3}u\|^2_{L^2(\mathbb{R}^{2d})}+\|\langle D_q\rangle^{2/3}u\|_{L^2(\mathbb{R}^{2d})}\Big)
\label{eq55}
\end{align}
holds for all $u\in \mathcal{C}_0^{\infty}(\mathbb{R}^{2d}).$
Here and throughout the paper we use the notation
\begin{align*}
\langle \cdot\rangle=\sqrt{1+|\cdot|^2}\;.
\end{align*}
Moreover we remind that for an arbitrary potential $V\in\mathcal{C}^{\infty}(\mathbb{R}^d)$, the Kramers-Fokker-Planck operator $K_V$ is essential maximal accretive when endowed with the domain $\mathcal{C}_0^{\infty}(\mathbb{R}^{2d})\;$(see Proposition 5.5, page 44 in \cite{HeNi}). Thanks to this property we deduce that the domain of the closure of $K_V$ is given by
\begin{align*}
D(K_V)=\left\{u\in L^2(\mathbb{R}^{2d}),\; K_Vu\in L^2(\mathbb{R}^{2d})\right\}~.
\end{align*}
Resultently, by density of $\mathcal{C}_0^{\infty}(\mathbb{R}^{2d})$
in the domain $D(K_V)$ all estimates written in this article, which
are verified with $C^\infty_0(\mathbb{R}^{2d})$ functions, can be
extended to $D(K_V).$ By relative bounded perturbation
  with bound less than $1$\,,
  this result holds as well when $V\in
  \mathcal{C}^{\infty}(\mathbb{R}\setminus\left\{0\right\})$ is an
  homogeneous function of degree $r>1$.

Our results will require the following assumption after setting \begin{align}
\mathcal{S}=\left\lbrace q\in\mathbb{R}^{d},\;\;|q|=1 \right\rbrace\;.\end{align}
\begin{assumption}\label{a.31.4}
The potential $V(q)$ is an homogeneous function of degree $r> 2$ in\\ $\mathcal{C}^{\infty}(\mathbb{R}^d\setminus\left\{~0\right\})$ and satisfies:
\begin{align}
\forall\;\; q\in\mathcal{S}\;,\;\;\;\;\;\;\;\;\;\partial_qV(q)=0\Rightarrow \mathrm{Tr}_{-,V}(q)>0\;.\label{a31.4.}
\end{align}
\end{assumption}
Our main result is the following.
\begin{thm}\label{a.3thm1.1}
If the potential $V(q)$ verifies Assumption \ref{a.31.4}, then
there exists a strictly positive constant $C_{V}>1$ (which depends on $V$) such that
\begin{align}
\|K_{V}u\|^2_{L^2}+C_{V}\|u\|^2_{L^2}\ge \frac{1}{C_V}\Big(\|L(O_p)u\|^2_{L^2}&+\|L(\langle\nabla V(q)\rangle^{\frac{2}{3}})  u\|^2_{L^2}\nonumber\\&+\|L(\langle\mathrm{Hess}\; V(q)\rangle^{\frac{1}{2}} ) u\|^2_{L^2}+\|L(\langle D_q\rangle^{\frac{2}{3}} ) u\|^2_{L^2}\Big)~,\label{a.31.6}
\end{align}
holds for all $u\in D(K_{V})$ where $L(s)=\frac{s+1}{\log(s+1)}$ for any $s\ge1.$ 
\end{thm}
\begin{cor}
\label{cor}
The Kramers-Fokker-Planck operator $K_V$ with a potential $V(q)$ satisfying Assumption \ref{a.31.4} has a compact resolvent.
\end{cor}
\begin{proof}
Let $0<\delta<1.$ Define the functions $f_\delta:\mathbb{R}^d\to\mathbb{R}$ by $$f_\delta(q)= |\nabla V(q)|^{\frac{4}{3}(1-\delta)}+|\mathrm{Hess}\,V(q)|^{1-\delta}~.$$
As a result of (\ref{a.31.6}) in Theorem~\ref{a.3thm1.1} there is a constant $C_V>1$ such that 
\begin{align*}
\|K_Vu\|^2_{L^2}+C_V\|u\|^2_{L^2}\ge\frac{1}{C_V}\Big(\langle u,f_\delta u\rangle+ \|L(O_p)u\|^2_{L^2}+\|L(\langle D_q\rangle^{\frac{2}{3}} ) u\|^2_{L^2}\Big)~,
\end{align*}
holds for all $u\in\mathcal{C}_0^{\infty}(\mathbb{R}^{2d})$ and all $\delta\in (0,1).$
In order to show that the operator $K_V$ has a compact resolvent it is sufficient to prove that $\lim\limits_{q\to +\infty}f_\delta(q)=+\infty.$
It is a matter of how different derivatives scale. Consider the unit sphere $S=\{q\in\mathbb{R}^{d}:|q|=1\}$. By Assumption (\ref{a31.4.}), at every point on $S$ either $\nabla V\neq 0$ or $|\mathrm{Hess}\; V|\neq 0$. Then the function $f_\delta$ is always positive on $S$. By hypothesis, $f_\delta$ is continuous on $S$ and therefore it achieves a positive minimum there, call it $m_\delta>0$.

For any $y,|y|>1$ there exists $\lambda>1$ such that $y=\lambda q$ for some $q\in S$. By homogeneity,
$$
V(y)=\lambda^rV\left(\frac{y}{\lambda}\right)=\lambda^rV(q)
$$
and therefore, by the chain rule
$$
|\nabla V(y)|=\lambda^{r-1}|\nabla V(q)|
$$
and
$$
|\mathrm{Hess}\;V(y)|=\lambda^{d(r-2)}|\mathrm{Hess}\;V(q)|.
$$
Adding these up,
$$
|\nabla V(y)|^{\frac{4}{3}(1-\delta)}+|\mathrm{Hess}\,V(y)|^{1-\delta}\ge \lambda^{(1-\delta)\min\{\frac{4}{3}(r-1),d(r-2)\}}f_\delta(q)\ge m_\delta\lambda^{(1-\delta)\min\{\frac{4}{3}(r-1),d(r-2)\}}
$$
which goes to infinity as $|y|=\lambda\to\infty$, since by assumption $r>2$.
\end{proof}
\begin{RQ}
The result of Corollary \label{cor} does not hold in the case of homogenous polynomial of degree 2 with degenerate Hessian. Indeed, we already know that in this case, the resolvent of the Kramers-Fokker-Planck operator $K_V$ is not compact since it is not as so for the Witten Laplacian (cf. Proposition 5.19 and Theorem 10.16 in \cite{HeNi}).
\end{RQ}
\begin{RQ}
Our results are in agreement with the results of Wei-Xi-Li \cite{Li}\cite{Li2} and those of Helffer-Nier on Witten Laplacian with homogeneous potential \cite{HeNi1}.
\end{RQ}
\section{Observations and first inequalities}
\subsection{Dyadic partition of unity}
In this paper, we make use of a locally finite dyadic partition of unity with respect to the position variable $q \in \mathbb{R}^d.$ Such a partition is described in the following Proposition. For  a detailed proof, we refer to \cite{BCD} (see page 59).
\begin{prop}
Let $\mathcal{C}$ be the shell $\left\lbrace x\in\mathbb{R}^{d},\;\;
  \frac{3}{4}< |x|<\frac{8}{3} \right\rbrace.$ There
exist radial functions $\chi$ and $\phi$ valued in the interval $[0,
1],$ belonging respectively to $\mathcal{C}^{\infty}_{0}(B(0,
\frac{4}{3}))$ and to $\mathcal{C}^{\infty}_{0}(\mathcal{C})$ such that
\begin{align*}
\forall x\in \mathbb{R}^d,\quad\quad\chi(x) +\sum_{j\ge0}\phi(2^{-j}
x) = 1\;,
\end{align*}  
\begin{align*}
\forall x\in \mathbb{R}^d\setminus\left\{0\right\},\quad\quad\sum_{j\in\mathbb{Z}}\phi(2^{-j}
x) = 1\;.
\end{align*}
\end{prop}
Setting for all $q\in\mathbb{R}^d,$ 
\begin{align*}&\chi_{-1}(q)=\frac{\chi(2q)}{\Big(\chi^2(2q)+\sum\limits_{j'\ge0}\phi^2(2^{-j'}
q)\Big)^{\frac{1}{2}}}=\frac{\chi(2q)}{\Big(\chi^2(2q)+\phi^2(q)\Big)^{\frac{1}{2}}}\;,\\&
\chi_{j}(q)=\frac{\phi(2^{-j}
q)}{\Big(\chi^2(2q)+\sum\limits_{j'\ge0}\phi^2(2^{-j'}
q)\Big)^{\frac{1}{2}}}\stackrel{\text{if}~j\leq  2}{=}
\;,
\frac{\phi(2^{-j}
q)}{\Big(\sum\limits_{j-1\leq j'\leq j+1}\phi^2(2^{-j'}
q)\Big)^{\frac{1}{2}}}
\end{align*}
we get a localy finite dyadic partition of unity
\begin{align}
\sum_{j\geq-1}\chi_j^2(q)=
  \tilde{\chi}_{-1}^{2}(2|q|) +\tilde{\chi}_{0}^{2}(2|q|)+\sum_{j\geq 0}\widetilde{\chi}^2(2^{-j}|q|)=1\label{a.32.1}
\end{align}
where for all $j\in\mathbb{N},$ the cutoff functions $
\widetilde{\chi}_{0},\tilde{\chi}$ and $\widetilde{\chi}_{-1}$ belong
respectively to $\mathcal{C}_0^{\infty}(\left]
  \frac{3}{4},\frac{8}{3}\right[)$,  $\mathcal{C}_0^{\infty}(\left] \frac{3}{4},\frac{8}{3}\right[)$  and $\mathcal{C}_0^{\infty}(\left] 0,{\frac{4}{3}}\right[).$
\begin{lem}\label{a.3lem2.2}
Let $V$ be in $\mathcal{C}^{\infty}(\mathbb{R}^{d}\setminus\left\{0\right\}).$ Consider the Kramers-Fokker-Planck operator $K_{V}$ defined as in (\ref{a.3eq1}). For a locally finite partition of unity $\sum\limits_{j\ge-1}\chi^2_j(q)=1$ one has 
\begin{align}
\|K_{V}u\|^2_{L^2(\mathbb{R}^{2d})}=\sum\limits_{j\ge-1}\|K_{V}(\chi_ju)\|^2_{L^2(\mathbb{R}^{2d})}-\|(p\partial_q\chi_j)u\|^2_{L^2(\mathbb{R}^{2d})}\label{a.32.4}~,
\end{align}
for all $u\in\mathcal{C}_0^{\infty}(\mathbb{R}^{2d}).$

In particular when the cutoff functions $\chi_j$ have the form (\ref{a.32.1}), there exists a uniform constant  $c>0$  so that  
\begin{align}
(1+4c)\|K_{V}u\|^2_{L^2(\mathbb{R}^{2d})}+c\|u\|^2_{L^2(\mathbb{R}^{2d})}\ge\sum\limits_{j\ge-1}\|K_{V}(\chi_ju)\|^2_{L^2(\mathbb{R}^{2d})},\label{a.32.5}
\end{align}
holds for all $u\in\mathcal{C}_0^{\infty}(\mathbb{R}^{2d}).$
\end{lem} 
\begin{proof}
The proof of the equality (\ref{a.32.4}) is detailed in \cite{Ben}. Now it remains to show the inequality (\ref{a.32.5}), after considering a locally finite dyadic partition of unity 
\begin{align}
\sum_{j\geq-1}\chi_j^2(q)=1\;,
\end{align}
where for all $j\in\mathbb{N},$ the cutoff functions $ \chi_{j}$ and $\chi_{-1}$ are respectively supported in the shell \\$\left\lbrace q\in\mathbb{R}^{d},\;\; 2^j\frac{3}{4}\le |q|\le 2^j\frac{8}{4} \right\rbrace$  and in the ball $B( 0,\frac{3}{4}).$

Since the partition is locally finite, for each index $j\ge-1$ there are finitely many $j'$ such that $(\partial_q\chi_j)\chi_{j'}$ is nonzero. 
Along these lines, there exists a uniform constant $c>0$ so that
\begin{align}
\sum\limits_{j\geq-1}\|(p\partial_q\chi_j)u\|^2_{L^2}&=\sum\limits_{j\geq-1}\sum\limits_{j'\geq-1}\|(p\partial_q\chi_j)\chi_{j'}u\|^2_{L^2}\nonumber\\&\le c\sum\limits_{j\geq-1}\frac{1}{(2^j)^2}\|p\chi_ju\|^2_{L^2}~,\label{a.32.6}
\end{align}
holds for all $u\in\mathcal{C}_0^{\infty}(\mathbb{R}^{2d}).$

On the other hand, for every $u\in\mathcal{C}_0^{\infty}(\mathbb{R}^{2d}),$
\begin{align}
c\sum\limits_{j\geq-1}\frac{1}{(2^j)^2}\|p\chi_ju\|^2_{L^2}\le 4c\,\|pu\|^2_{L^2}\le8c\,\mathrm{Re}\;\langle u,K_Vu\rangle\le 4c\,(\|u\|^2_{L^2}+\|K_Vu\|^2_{L^2})\;.\label{a.32.7}
\end{align}
Collecting the estimates (\ref{a.32.4}), (\ref{a.32.6}) and (\ref{a.32.7}), we establish the desired inequality (\ref{a.32.5}). 
\end{proof}
\subsection{Localisation in a fixed Shell} 

\begin{lem}\label{a.3lem2.3}
Let $V(q)$ be an homogeneous function in $\mathcal{C}^{\infty}(\mathbb{R}^{d}\setminus\left\{0\right\})$ of degree $r$ and assume $j\in\mathbb{Z}.$ Given $u_j\in\mathcal{C}_0^{\infty}(\mathbb{R}^{2d}),$ one has 
\begin{align*}
\|K_{V}u_j\|_{L^2(\mathbb{R}^{2d})}=\|K_{j,V}v_j\|_{L^2(\mathbb{R}^{2d})}\;,
\end{align*}
where the operator $K_{j,V}$ is defined by \begin{align}
 K_{j,V}=\frac{1}{2^j}p\partial_q-(2^j)^{r-1}\partial_qV(q)\partial_p+O_p\;,\label{a.3222}\end{align}
and $\;v_j(q,p)=2^{\frac{jd}{2}}u_j(2^jq,p).$ 

In particular when $u_j$ is supported in  $\left\lbrace q\in\mathbb{R}^{d},\; 2^j\frac{3}{4}\le |q|\le 2^j\frac{8}{3}\right\rbrace,$ the support of $v_j$ is a fixed shell $\overline{\mathcal{C}}=\left\lbrace q\in\mathbb{R}^{d},\; \frac{3}{4}\le |q|\le \frac{8}{3}\right\rbrace\;.$ 
\end{lem}
\begin{proof}
Let $j\in\mathbb{Z}$ be an index. Assume $u_j\in\mathcal{C}_0^{\infty}(\mathbb{R}^{2d})$ and state \begin{align}
v_j(q,p)=2^{\frac{jd}{2}}u_j(2^jq,p)\;.\label{a.3211}
\end{align}
On the grounds that the function $V$ is homogeneous of degree $r$ we deduce that respectively its gradient $\partial_qV(q)$ is homogeneous of degree $r-1.$ As follows, we can write 
\begin{align*}
K_Vu_j(q,p)&=K_V\Big(2^{\frac{-jd}{2}}v_j(2^{-j}q,p)\Big)\\&=2^{\frac{-jd}{2}}\Big((2^{-j}p\partial_q-(2^j)^{r-1}\partial_qV(q)\partial_p+O_p)v_j\Big)(2^{-j}q,p)\;.
\end{align*}
Notice that if
\begin{align*}
\mathrm{supp}\;u_j\subset\left\lbrace q\in\mathbb{R}^{d},\; 2^j\frac{3}{4}\le |q|\le 2^j\frac{8}{3}\right\rbrace\;,
\end{align*}
the cutoff functions $v_j,$ defined in (\ref{a.3211}), are all supported in the fixed shell \begin{align*}
\overline{\mathcal{C}}=\left\lbrace q\in\mathbb{R}^{d},\; \frac{3}{4}\le |q|\le \frac{8}{3}\right\rbrace\;.
\end{align*}
\end{proof}
\begin{RQ}
Assume $j\in\mathbb{N}.$ If we introduce a small parameter $h=2^{-2(r-1)j}$ then the operator $K_{j,V},$ defined in (\ref{a.3222}), can be rewritten as
\begin{align*}
K_{j,V}=\frac{1}{h}\Big(\sqrt{h}p(h^{\frac{1}{2}+\frac{1}{2(r-1)}}\partial_q)-\sqrt{h}\partial_qV(q)\partial_p+\frac{h}{2}(-\Delta_p+p^2)\Big)\;.
\end{align*}
Now owing to a dilation with respect to the velocity variable $p,$ which for $(\sqrt{h}p,\sqrt{h}\partial_p)$ associates $(p,h\partial_p),$ we deduce that the operator $K_{j,V}$ is unitary equivalent to
\begin{align*}
\widehat{K}_{j,V}=\frac{1}{h}\Big(p(h^{\frac{1}{2}+\frac{1}{2(r-1)}}\partial_q)-\partial_qV(q)h\partial_p+\frac{1}{2}(-h^2\Delta_p+p^2)\Big)\;.
\end{align*}
In particular, taking $r=2,$ 
\begin{align*}
\widehat{K}_{j,V}=\frac{1}{h}\Big(p(h\partial_q)-\partial_qV(q)h\partial_p+\frac{1}{2}(-h^2\Delta_p+p^2)\Big)\;,
\end{align*}
is clearly a semiclassical operator with respect to the variables $q$ and $p$.
However if $r>2$, the operator $\widehat{K}_{j,V}$ is semiclassical only with respect to the velocity variable  $p$ (since $h^{\frac{1}{2}+\frac{1}{2(r-1)}}>h$).
For a polynomial $V(q),$ the case $r=2$ corresponds to  the quadratic situation. Extensive works have been done concerned with this case (see  \cite{Hor}\cite{HiPr}\cite{Vio}\cite{Vio1}\cite{AlVi}\cite{BNV}).   
\end{RQ}

\section{Proof of the main result}
In this section we present the proof of Theorem~\ref{a.3thm1.1}.
\begin{proof}
In the whole proof we denote \begin{align*}
\overline{\mathcal{C}}=\left\lbrace q\in\mathbb{R}^{d},\; \frac{3}{4}\le |q|\le \frac{8}{3}\right\rbrace\;.
\end{align*}
Assume $u\in\mathcal{C}_0^{\infty}(\mathbb{R}^{2d})$ and consider a localy finite dyadic partition of unity defined as in (\ref{a.32.1}). By Lemma \ref{a.3lem2.2} (see (\ref{a.32.5})), there is a uniform constant $c$ such that
\begin{align}
(1+4c)\|K_{V}u\|^2_{L^2(\mathbb{R}^{2d})}+c\|u\|^2_{L^2(\mathbb{R}^{2d})}\ge\sum\limits_{j\ge-1}\|K_{V}u_j\|^2_{L^2(\mathbb{R}^{2d})}.\label{a.33111}
\end{align}
where we denote $u_j=\chi_ju.$ We obtain by Lemma \ref{a.3lem2.3} and the estimate (\ref{a.33111})
\begin{align}
(1+4c)\|K_{V}u\|^2_{L^2(\mathbb{R}^{2d})}+c\|u\|^2_{L^2(\mathbb{R}^{2d})}\ge\sum\limits_{j\ge-1}\|K_{j,V}v_j\|^2_{L^2(\mathbb{R}^{2d})}\;,
\end{align}
where the operator 
\begin{align*}
 K_{j,V}=\frac{1}{2^j}p\partial_q-(2^j)^{r-1}\partial_qV(q)\partial_p+O_p\;,\end{align*}
and $ v_j(q,p)=2^{\frac{jd}{2}}u_j(2^jq,p)\;.$ Setting $h=2^{-2(r-1)j},$ one has
\begin{align*}
 K_{j,V}=p(h^{\frac{1}{2(r-1)}}\partial_q)-h^{-\frac{1}{2}}\partial_qV(q)\partial_p+\frac{1}{2}(-\Delta_p+p^2)\;.\end{align*}
Now, fix $\nu>0$ such that 
\begin{align}
 \max(\frac{1}{6},\frac{1}{8}+\frac{3}{8(r-1)})<\nu<\frac{1}{4}+\frac{1}{4(r-1)}~. \label{a.32..9}
\end{align}
Such a choice is always possible: 
\begin{itemize}
\item In the case $r\ge 10,$
$\max(\frac{1}{6},\frac{1}{8}+\frac{3}{8(r-1)})$ equals $\frac{1}{6}$
while $\frac{1}{4}+\frac{1}{4(r-1)}$ is always greater than
$\frac{1}{4}.$ So we can choose a value $\nu$ independent of $r$
between $\frac{1}{6}$ and $\frac{1}{4}.$
\item  in the case $2< r<10,$ $\max(\frac{1}{6},\frac{1}{8}+\frac{3}{8(r-1)})$ equals $\frac{1}{8}+\frac{3}{8(r-1)}<\frac{1}{4}+\frac{1}{4(r-1)}$\,. Hence, we can choose for example $\nu=\frac{3}{16}+\frac{5}{16(r-1)}.$ 
\end{itemize}
Taking  $\nu>0,$ satisfying (\ref{a.32..9}), we consider a locally finite partition of unity with respect to $q\in\mathbb{R}^d$ given by
\begin{align*}
\sum\limits_{k\ge-1}(\theta_{k,h}(q))^2&=\sum\limits_{k\ge-1}\Big(\theta(\frac{1}{|\ln(h)|h^{\nu}}q-q_k)\Big)^2\nonumber\\&=\sum\limits_{k\ge-1}\Big(\theta(\frac{1}{|\ln(h)|h^{\nu}}(q-q_{k,h})\Big)^2=1\;,
\end{align*}
where for any index $k$ \begin{align*}
q_{k,h}=|\ln(h)|h^{\nu}q_{k}\;,\;\;\;\;\;\mathrm{supp}\;\theta_{k,h}\subset B(q_{k,h},|\ln(h)|h^{\nu})\;,\;\;\;\;\;\theta_{k,h}\equiv1\;\;\text{in}\;\;B(q_{k,h},\frac{1}{2}|\ln(h)|h^{\nu})\;.
\end{align*}
Using this partition we get through Lemma \ref{a.3lem2.2} (see (\ref{a.32.4})),
\begin{align}
\|K_{j,V}v_j\|^2_{L^2}\ge  \sum\limits_{k\ge-1}\|K_{j,V}\theta_{k,h}v_j\|^2_{L^2}-|\ln(h)|^{-2}h^{\frac{1}{r-1}-2\nu}\|p\theta_{k,h}v_j\|^2_{L^2}\;.\label{a.3322}
\end{align}
In order to reduce the written expressions we denote in the whole of the proof \begin{align*}w_{k,j}=\theta_{k,h}v_j\;.\end{align*}
Taking into account (\ref{a.3322}),
\begin{align}
\|K_{j,V}v_j\|^2_{L^2}&\ge  \sum\limits_{k\ge-1}\|K_{j,V}w_{k,j}\|^2_{L^2}-|\ln(h)|^{-2}h^{\frac{1}{r-1}-2\nu}\|w_{k,j}\|_{L^2}\|K_{j,V}w_{k,j}\|_{L^2}\nonumber\\&\ge\sum\limits_{k\ge-1}{\frac{3}{4}}\|K_{j,V}w_{k,j}\|^2_{L^2}-2|\ln(h)|^{-2}h^{\frac{1}{r-1}-2\nu}\|w_{k,j}\|^2_{L^2}\;.\label{a.32.13}
\end{align}
Notice that in the last inequality we  simply use respectively the fact that 
\begin{align*}
\|pw_{k,j}\|^2_{L^2}\leq 2\mathrm{Re}\langle w_{k,j},K_{j,V}w_{k,j}\rangle\le \|w_{k,j}\|_{L^2}\|K_{j,V}w_{k,j}\|_{L^2}\;,
\end{align*}
and the Cauchy inequality with epsilon ( $ab\le \epsilon a^2+\frac{1}{4\epsilon}b^2$).

From now on, set \begin{align*}
K_0=\left\{q\in \overline{\mathcal{C}}\;,\;\;\;\partial_qV(q)=0\right\}\;.
\end{align*}
Clearly, by continuity of the map $q\mapsto\partial_qV(q)$ on the
shell $\overline{\mathcal{C}}$ (which is a compact set of $\mathbb{R}^d$), we
deduce the compactness of $K_0.$

Since $q\mapsto\frac{\mathrm{Tr}_{-,V}(q)}{1+\mathrm{Tr}_{+,V}(q)}$ is
uniformly continuous on any compact neighborhood of $K_{0}$\,, there
exists $\varepsilon_{1}>0$ such that
\begin{align}
d(q,K_0)\le \epsilon_1\Rightarrow \frac{\mathrm{Tr}_{-,V}(q)}{1+\mathrm{Tr}_{+,V}(q)}\ge \frac{\epsilon_0}{2}\;,\label{a.33.6.}
\end{align}
where $\epsilon_0:=\min\limits_{q\in K_0}\frac{\mathrm{Tr}_{-,V}(q)}{1+\mathrm{Tr}_{+,V}(q)}.$

On the other hand, in vue of the definition of $K_0$ and by continuity of $q\mapsto \partial_qV(q)$ on $\overline{\mathcal{C}},$ there is a constant $\epsilon_2>0$ (that depends on $\epsilon_1$) such that 
\begin{align}
\forall\; q\in \overline{\mathcal{C}}\;,\;\;d(q,K_0)\ge \epsilon_1\Rightarrow |\partial_qV(q)|\ge \epsilon_2\;.\label{a.337}
\end{align}
Now let us introduce
\begin{align*}
\Sigma(\epsilon_1)=\left\{q\in \mathcal{C}\;,\;\;d(q,K_0)\ge\epsilon_1\right\}\;,
\end{align*}
\begin{align*}
I(\epsilon_1)=\left\{k\in \mathbb{Z}\;,\;\;\mathrm{supp}\;\theta_{k,h}\subset\Sigma(\epsilon_1)\right\}\;.
\end{align*}
In order to establish a subelliptic estimate for $K_{j,V},$ we distinguish the two following cases.
\begin{outerdesc}
    \item[Case \textbf{1}] $k\not\in I(\epsilon_1).$ In this case the support of the cutoff function $\theta_{k,h}$ might intercect the set of zeros of the gradient of $V.$
    \item[Case \textbf{2}] $k\in I(\epsilon_1).$ Here the gradient of $V$ does not vanish for all $q$ in the support of $\theta_{k,h}.$
\end{outerdesc}
\smallskip
    
    The idea is to use, in the suitable situation, either  quadratic
    or linear approximating polynomial $\widetilde{V}$ near some $q'_{k,h}\in\mathrm{supp}\;\theta_{k,h}$ to write 
\begin{align*}
\sum\limits_{k\ge-1}\|K_{j,V}w_{k,j}\|^2_{L^2}\ge \frac{1}{2}\sum\limits_{k\ge-1}\|K_{j,\widetilde{V}}w_{k,j}\|^2_{L^2}-\|(K_{j,V}-K_{j,\widetilde{V}})w_{k,j}\|^2_{L^2}\;,
\end{align*}
or equivalently
\begin{align}
\sum\limits_{k\ge-1}\|K_{j,V}w_{k,j}\|^2_{L^2}\ge \frac{1}{2}\sum\limits_{k\ge-1}\|K_{j,\widetilde{V}}w_{k,j}\|^2_{L^2}-\|\frac{1}{\sqrt{h}}(\partial_qV(q)- \partial_q\widetilde{V}(q))\partial_pw_{k,j}\|^2_{L^2}\;.\label{a.33.444}
\end{align}
Then based on the estimates written in \cite{BNV}, which are valid for the operator $K_{\widetilde{V}},$ we deduce a subelliptic estimate for $K_{\widetilde{V}},$ after a careful control of the errors which appear in (\ref{a.32.13}) and (\ref{a.33.444}).

\noindent\textbf{Case 1.} In this situation, we use the quadractic approximation near some element \\$q'_{k,h}\in\mathrm{supp}\;\theta_{k,h}\cap(\mathbb{R}^d\setminus \Sigma(\epsilon_1)),$
\begin{align*}
V^2_{k,h}(q)&=\sum\limits_{|\alpha|\le2}\frac{\partial_q^{\alpha}V(q'_{k,h})}{\alpha!}(q-q'_{k,h})^{\alpha}\;.
\end{align*}
Notice that one has for all $q\in\mathbb{R}^d,$
\begin{align}
 |V(q)- V^2_{k,h}(q)|=\mathcal{O}(|q-q'_{k,h}|^3)\;.
\end{align}
Accordingly, for every $q$ in the support of $w_{k,j},$
 \begin{align}
 |\partial_qV(q)- \partial_qV^2_{k,h}(q)|&=\mathcal{O}(|q-q'_{k,h}|^2)\nonumber\\&=\mathcal{O}(|\ln(h)|^2h^{2\nu})\;.\label{a.33.66}
\end{align}
Combining (\ref{a.33.444}) and (\ref{a.33.66}), there is  a constant $c > 0$ such that 
 \begin{align}
\sum\limits_{k\ge-1}\|K_{j,V}w_{k,j}\|^2_{L^2}&\ge \frac{1}{2}\sum\limits_{k\ge-1}\|K_{j,V^2_{k,h}}w_{k,j}\|^2_{L^2}-c\,\frac{(|\ln(h)|^2h^{2\nu})^2}{h}\|\partial_pw_{k,j}\|^2_{L^2}\nonumber\\&\ge \frac{1}{2}\sum\limits_{k\ge-1}\|K_{j,V^2_{k,h}}w_{k,j}\|^2_{L^2}-c\,\frac{(|\ln(h)|^2h^{2\nu})^2}{h}\|w_{k,j}\|_{L^2}\|K_{j,V^2_{k,h}}w_{k,j}\|_{L^2}\nonumber\\&\ge \frac{3}{16}\sum\limits_{k\ge-1}\|K_{j,V^2_{k,h}}w_{k,j}\|^2_{L^2}-2c\,\frac{(|\ln(h)|^2h^{2\nu})^2}{h}\|w_{k,j}\|^2_{L^2}\;.\label{a.32.77}
\end{align}
Putting (\ref{a.32.13}) and (\ref{a.32.77}) together,
\begin{align}
\|K_{j,V}v_j\|^2\ge \frac{9}{64}\sum\limits_{k\ge-1}\|K_{j,V^2_{k,h}}w_{k,j}\|^2-\frac{3}{2}\,c\,\frac{(|\ln(h)|^2h^{2\nu})^2}{h}\|w_{k,j}\|^2-2|\ln(h)|^{-2}h^{\frac{1}{r-1}-2\nu}\|w_{k,j}\|^2\;.\label{a.3312}
\end{align} 
On the other hand, owning to a change of variables $q"=qh^{\frac{1}{2(r-1)}},$ one can write 
\begin{align}
\|K_{j,V^2_{k,h}}w_{k,j}\|_{L^2}=\|\widetilde{K}_{j,V^2_{k,h}}\widetilde{w}_{k,j}\|_{L^2}\label{a.33.88}\;,
\end{align}
where the operator $\widetilde{K}_{j,V^2_{k,h}}$ reads
\begin{align*}
\widetilde{K}_{j,V^2_{k,h}}&=p\partial_q-h^{-\frac{1}{2}}\partial_qV^2_{k,h}(h^{\frac{1}{2(r-1)}}q)\partial_p+\frac{1}{2}(-\Delta_p+p^2)
\\
&
=
p\partial_q-\underbrace{h^{-\frac{1}{2}+\frac{1}{2(r-1)}}}_{=:H}\partial_qV^2_{k,h}(q)\partial_p+\frac{1}{2}(-\Delta_p+p^2)
\;,
\end{align*}
and 
\begin{align*}
w_{k,j}(q,p)=\frac{1}{h^{\frac{d}{4(r-1)}}}\widetilde{w}(\frac{q}{h^{\frac{1}{2(r-1)}}},p)\;.
\end{align*}
In the rest of the proof we denote \begin{align*}
H=h^{-\frac{1}{2}}h^{\frac{1}{2(r-1)}}\;.
\end{align*}
From now on assume $j\in\mathbb{N}.$ In view of (\ref{a.33.6.}), $\mathrm{Tr}_{-,V^2_{k,h}}=\mathrm{Tr}_{-,V}(q'_{k,h})\not=0.$ Hence by (\ref{1.5mm}), 
\begin{align}
\|\widetilde{K}_{j,V^2_{k,h}}\widetilde{w}_{k,j}\|^2_{L^2}\ge c\,\frac{1+H\mathrm{Tr}_{-,V^2_{k,h}}}{\log(2+H\mathrm{Tr}_{-,V^2_{k,h}})^2}\| \widetilde{w}_{k,j}\|^2_{L^2}\;.
\end{align}
Or samely 
\begin{align}
\|\widetilde{K}_{j,V^2_{k,h}}\widetilde{w}_{k,j}\|^2_{L^2}\ge c\,\frac{1+H\mathrm{Tr}_{-,V}(q'_{k,h})}{\log(2+H\mathrm{Tr}_{-,V}(q'_{k,h}))^2}\| \widetilde{w}_{k,j}\|^2_{L^2}\;.\label{a.3316}
\end{align}
Using once more (\ref{a.33.6.}), 
\begin{align}
\mathrm{Tr}_{-,V}(q'_{k,h})\ge \frac{\epsilon_0}{2}(1+\mathrm{Tr}_{+,V}(q'_{k,h}))\;,
\end{align}
where we remind that $\epsilon_0=\min\limits_{q\in K_0}\frac{\mathrm{Tr}_{-,V}(q)}{1+\mathrm{Tr}_{+,V}(q)}.$ Consequently
\begin{align}
|\mathrm{Hess}\;V(q'_{k,h})|\ge \mathrm{Tr}_{-,V}(q'_{k,h})\ge \frac{\epsilon_0}{2}\;,
\end{align}
and 
\begin{align}
\mathrm{Tr}_{-,V}(q'_{k,h})&\ge \frac{1}{2}\mathrm{Tr}_{-,V}(q'_{k,h})+\frac{\epsilon_0}{4}(1+\mathrm{Tr}_{+,V}(q'_{k,h}))\nonumber\\&\ge \frac{1}{2}\min(1,\frac{\epsilon_0}{2})(\mathrm{Tr}_{-,V}(q'_{k,h})+\mathrm{Tr}_{+,V}(q'_{k,h}))\nonumber\\&\ge \frac{1}{2}\min(1,\frac{\epsilon_0}{2})|\mathrm{Hess}\;V(q'_{k,h})|\;.\label{a.3318}
\end{align} 
Furthermore by continuity of the map $q\mapsto \mathrm{Tr}_{-,V}(q)$ on the compact set $
\overline{\mathcal{C}},$ there exists a constant $\epsilon_3 > 0$ such that 
$\mathrm{Tr}_{-,V}(q)\le \epsilon_3$ for all $q\in\overline{\mathcal{C}}.$ Hence \begin{align}\frac{\epsilon_0}{2}\le \mathrm{Tr}_{-,V}(q'_{k,h})\le \epsilon_3\;.\label{a.3317}
\end{align}
From (\ref{a.3316}), (\ref{a.3318}) and  (\ref{a.3317}),
\begin{align*}
\|\widetilde{K}_{j,V^2_{k,h}}\widetilde{w}_{k,j}\|^2_{L^2}\ge c\,\frac{H}{\log(H)^2}\| \widetilde{w}_{k,j}\|^2_{L^2}\;.
\end{align*}
It follows from the above inequality and (\ref{a.33.88}),
\begin{align}
\|K_{j,V^2_{k,h}}w_{k,j}\|^2_{L^2}\ge c\,\frac{H}{\log(H)^2}\| w_{k,j}\|^2_{L^2}\;.\label{a.33.1333}
\end{align}
Now using the estimate (\ref{a.33.1333}), we should control the errors
coming from the partition of unity and the quadratic
approximation. For this reason, notice that  our choice of
exponent $\nu$ in  \eqref{a.32..9} implies
\begin{align*}
\left\{
\begin{array}{l}
\frac{(|\ln(h)|^2h^{2\nu})^2}{h}\ll \frac{H}{\log(H)^2}\\
 |\ln(h)|^{-2}h^{\frac{1}{r-1}-2\nu}\ll \frac{H}{\log(H)^2}\;.
\end{array}
\right.
\end{align*}
As a result, collectting the estimates (\ref{a.3312}) and (\ref{a.33.1333}), we deduce the existence of a constant $c>0$ such that 
\begin{align}
\|K_{j,V}v_j\|^2_{L^2}\ge c\sum\limits_{k\ge-1}\|K_{j,V^2_{k,h}}w_{k,j}\|^2_{L^2}\;.\label{a.33...15}
\end{align}
Via (\ref{eq44}), there is a constant $c>0$ so that
\begin{align}
\|\widetilde{K}  _{j,V^2_{k,h}}\widetilde{w}_{k,j}\|^2+(1+10c)H|\mathrm{Hess}\;V(q'_{k,h})|\|\widetilde{w}_{k,j}\|^2\ge c\Big(\|O_p\widetilde{w}_{k,j}\|^2&+\|\langle D_q \rangle^{\frac{2}{3}} \widetilde{w}_{k,j}\|^2\nonumber\\&+H|\mathrm{Hess}\;V(q'_{k,h})|\|\widetilde{w}_{k,j}\|^2\Big)\;.
\end{align}
Hence using the reverse change of variables $q"=\frac{q}{h^{\frac{1}{2(r-1)}}}\;,$ we obtain in view of the above estimate and (\ref{a.33.88}),
\begin{align}
\|K_{j,V^2_{k,h}}w_{k,j}\|^2+(1+10c)H|\mathrm{Hess}\;V(q'_{k,h})|\|w_{k,j}\|^2\ge c\Big(\|O_pw_{k,j}\|^2&+\|\langle h^{\frac{1}{2(r-1)}} D_q \rangle^{\frac{2}{3}} w_{k,j}\|^2\nonumber\\&+H|\mathrm{Hess}\;V(q'_{k,h})|\|w_{k,j}\|^2\Big)\;.\label{a.33.16}
\end{align}
Or by (\ref{a.3318}) and  (\ref{a.3317}),
\begin{align}
\frac{\epsilon_0}{2}\le|\mathrm{Hess}\;V(q'_{k,h})|\le \frac{2\epsilon_3}{\min(1,\frac{\epsilon_0}{2})}\;,\label{a.33.19}\end{align}
Putting (\ref{a.33.16}) and (\ref{a.33.19}) together, there is a constant $c>0$ so that
\begin{align}
\|K_{j,V^2_{k,h}}w_{k,j}\|^2+H\|w_{k,j}\|^2\ge c\Big(\|O_pw_{k,j}\|^2&+\|\langle h^{\frac{1}{2(r-1)}} D_q \rangle^{\frac{2}{3}}w_{k,j}\|^2\nonumber\\&+H\|w_{k,j}\|^2+\|\langle H|\mathrm{Hess}\;V(q'_{k,h})|\rangle^{\frac{1}{2}}w_{k,j}\|^2\Big)\;.\label{a
33.1888}
\end{align}
On the other hand, for all $q\in\mathrm{supp}\;w_{k,j},$ \begin{align}|\mathrm{Hess}\;V(q)-\mathrm{Hess}\;V(q'_{k,h})|=\mathcal{O}(|q-q'_{k,h}|)=\mathcal{O}(|\ln(h)|h^{\nu})\label{a.33266}\;\end{align}
Therefore by (\ref{a.33.19}) and (\ref{a.33266}), we obtain for every $q\in\mathrm{supp}\;w_{k,j}$ and all $j$ sufficiently large.
\begin{align}
 \frac{1}{2}|\mathrm{Hess}\;V(q'_{k,h})|\le|\mathrm{Hess}\;V(q)|\le \frac{3}{2}|\mathrm{Hess}\;V(q'_{k,h})|\;.\label{a.3326}
\end{align}
From (\ref{a
33.1888}) and (\ref{a.3326}), there exists a constant $c>0$ so that 
\begin{align}
\|K_{j,V^2_{k,h}}w_{k,j}\|^2+H\|w_{k,j}\|^2\ge c\Big(\|O_pw_{k,j}\|^2&+\|\langle h^{\frac{1}{2(r-1)}} D_q \rangle^{\frac{2}{3}} w_{k,j}\|^2\nonumber\\&+H\|w_{k,j}\|^2+\|\langle H|\mathrm{Hess}\;V(q)|\rangle^{\frac{1}{2}}w_{k,j}\|^2\Big)\;,\label{a
3328}
\end{align}
is valid for all $j$ large enough.

Furthermore, by continuity of the map $q\mapsto|\partial_qV(q)|^{\frac{4}{3}}$ on the fixed shell $\overline{\mathcal{C}}$, for all $q\in\mathrm{supp}\; w_{k,j}$
\begin{align}
\frac{1}{4}H\ge c\,|h^{-\frac{1}{2}}\partial_qV(q)|^{\frac{4}{3}}\;,\label{a
33.199}
\end{align}
holds for all $j$ sufficiently large.

In such a way, considering (\ref{a
3328}) and (\ref{a
33.199})
\begin{align}
\|K_{j,V^2_{k,h}}w_{k,j}\|^2+(2+H)\|w_{k,j}\|^2\ge c\Big(\|&O_pw_{k,j}\|^2+\|\langle h^{\frac{1}{2(r-1)}} D_q \rangle^{\frac{2}{3}} w_{k,j}\|^2+(2+H)\|w_{k,j}\|^2\nonumber\\&+\|(H|\mathrm{Hess}\;V(q)|)^{\frac{1}{2}}w_{k,j}\|^2+\|\langle h^{-\frac{1}{2}}|\partial_qV(q)|\rangle^{\frac{2}{3}}w_{k,j}\|^2\Big)\;.\label{a.33.20000}
\end{align}
Putting (\ref{a.33.1333}) and (\ref{a.33.20000}) together,
\begin{align}
\|K_{j,V^2_{k,h}}w_{k,j}\|^2\ge c\Big(\|\frac{O_p}{\log(2+H)}&w_{k,j}\|^2+\|\frac{\langle h^{\frac{1}{2(r-1)}} D_q \rangle^{\frac{2}{3}}}{\log(2+H)}w_{k,j}\|^2+\|\frac{(2+H)^{\frac{1}{2}}}{\log(2+H)}w_{k,j}\|^2\nonumber\\&+\|\frac{\langle H|\mathrm{Hess}\;V(q)|\rangle^{\frac{1}{2}}}{\log(2+H)}w_{k,j}\|^2+\|\frac{\langle h^{-\frac{1}{2}}|\partial_qV(q)|\rangle^{\frac{2}{3}}}{\log(2+H)}w_{k,j}\|^2\Big)\;,\label{a.33.201}
\end{align}
holds for all $j\ge j_0,$ for some $j_0\ge 1$ large enough.

Now let us collect the finite remaining terms for $-1\le j\le j_0.$
After recalling $h=2^{-j}$ and $H=h^{-\frac{1}{2}+\frac{1}{2(r-1)}}$
we define
\begin{multline*}
   c_V^{(1)}=\max_{-1\le j\le j_0}
\left[A_{V_{k,h}^{2}}+\sup_{q\in \mathrm{supp}\,(\chi_{j}\theta_{k,h})}\left(\langle
H|\mathrm{Hess}~V(q)|\rangle+\langle
h^{-\frac{1}{2}}\left|\partial_{q}V(q)\right|\rangle^{4/3}\right)
\right.
\\
\left.
+\frac{(2+H)}{\log(2+H)^2}+\frac{3}{2}\,c\,\frac{(|\ln(h)|^2h^{2\nu})^2}{h}+2|\ln(h)|^{-2}h^{\frac{1}{r-1}-2\nu}\right]\,.
\end{multline*}
From the lower bound \eqref{eq44}, we deduce the existence of a constant $c>0$ so that 
\begin{align}
\frac{9}{64}\|K_{V_{k,h}^{2}}w_{k,j}\|&+(c_V^{(1)}-\frac{3}{2}\,c\,\frac{(|\ln(h)|^2h^{2\nu})^2}{h}-2|\ln(h)|^{-2}h^{\frac{1}{r-1}-2\nu})\|w_{k,j}\|^2\nonumber\\&\geq c\Big(\|O_pw_{k,j}\|^{2}+\|\langle
 h^{\frac{1}{2(r-1)}} D_{q}\rangle^{2/3}w_{k,j}\|^{2} 
+\|\langle h^{-\frac{1}{2}}|\partial_{q}V(q)|\rangle^{2/3}w_{k,j}\|^{2}
\nonumber\\&\hspace{4cm}+\|\langle 
H|\mathrm{Hess}~V(q)|\rangle^{1/2}w_{k,j}\|^{2}
+\|\frac{(2+H)^{\frac{1}{2}}}{\log(2+H)}w_{k,j}\|^{2}\Big)\;,\label{a.33334}
\end{align}
holds for all $-1\le j\le j_0.$

Finally, collecting (\ref{a.33...15}), (\ref{a.33.201}) and (\ref{a.33334}),
\begin{align}
\|K_{j,V}v_j\|^2+c_V^{(2)}\|v_j\|^2\ge c\sum\limits_{k\not\in I(\epsilon_1)}\Big(\|&\frac{O_p}{\log(2+H)}w_{k,j}\|^2+\|\frac{\langle h^{\frac{1}{2(r-1)}} D_q \rangle^{\frac{2}{3}}}{\log(2+H)}w_{k,j}\|^2+\|\frac{(2+H)^{\frac{1}{2}}}{\log(2+H)}w_{k,j}\|^2\nonumber\\&+\|\frac{\langle H|\mathrm{Hess}\;V(q)|\rangle^{\frac{1}{2}}}{\log(2+H)}w_{k,j}\|^2+\|\frac{\langle h^{-\frac{1}{2}}|\partial_qV(q)|\rangle^{\frac{2}{3}}}{\log(2+H)}w_{k,j}\|^2\Big)\;,\label{a.33.2000}
\end{align}
is valid for every $j\ge -1.$ 
\smallskip

\noindent\textbf{Case 2.} We consider in this case the linear approximating polynomial
\begin{align*}
V^1_{k,h}(q)&=\sum\limits_{|\alpha|=1}\frac{\partial_q^{\alpha}V(q_{k,h})}{\alpha!}(q-q_{k,h})^{\alpha}\;.
\end{align*}
Note that for any $q\in\mathbb{R}^d,$
\begin{align}
 |V(q)- V^1_{k,h}(q)|=\mathcal{O}(|q-q_{k,h}|^2)\;,
\end{align}
and for every $q\in\mathrm{supp}\;w_{k,j},$
 \begin{align}
 |\partial_qV(q)- \partial_qV^1_{k,h}(q)|&=\mathcal{O}(|q-q_{k,h}|)\nonumber\\&=\mathcal{O}(|\ln(h)|h^{\nu})\;.\label{a.32.24}
\end{align}
Due to (\ref{a.33.444}) and (\ref{a.32.24}),
 \begin{align}
\sum\limits_{k\ge-1}\|K_{j,V}w_{k,j}\|^2&
\ge \frac{1}{2}\sum\limits_{k\ge-1}\|K_{j,V^1_{k,h}}w_{k,j}\|^2-c\,\frac{(|\ln(h)|h^{\nu})^2}{h}\|\partial_pw_{k,j}\|^2\nonumber\\&\ge \frac{1}{2}\sum\limits_{k\ge-1}\|K_{j,V^1_{k,h}}w_{k,j}\|^2-c\,\frac{(|\ln(h)|h^{\nu})^2}{h}\|w_{k,j}\|\|K_{j,V^1_{k,h}}w_{k,j}\|\nonumber\\&\ge \frac{3}{16}\sum\limits_{k\ge-1}\|K_{j,V^1_{k,h}}w_{k,j}\|^2-2c\,\frac{(|\ln(h)|h^{\nu})^2}{h}\|w_{k,j}\|^2\;.\label{a.32.255}
\end{align}
Assembling (\ref{a.32.13}) and (\ref{a.32.255}),
\begin{align}
\|K_{j,V}v_j\|^2\ge \frac{9}{64}\sum\limits_{k\ge-1}\|K_{j,V^1_{k,h}}w_{k,j}\|^2-\frac{3}{2}\,c\,\frac{(|\ln(h)|^2h^{2\nu})^2}{h}\|w_{k,j}\|^2-2|\ln(h)|^{-2}h^{\frac{1}{r-1}-2\nu}\|w_{k,j}\|^2\;.\label{a.3336}
\end{align} 
Additionally, one has
\begin{align}
\|K_{j,V^1_{k,h}}w_{k,j}\|_{L^2}=\|\widetilde{K}_{j,V^1_{k,,h}}\widetilde{w}_{k,j}\|_{L^2}\;,\label{a.32.666}
\end{align}
where the operator $\widetilde{K}_{j,V^1_{k,,h}}$ is given by
\begin{align}
\widetilde{K}_{j,V^1_{k,h}}&=p\partial_q-h^{-\frac{1}{2}}\partial_qV^1_{k,h}(h^{\frac{1}{2(r-1)}}q)\partial_p+\frac{1}{2}(-\Delta_p+p^2)\nonumber\\&=p\partial_q-h^{-\frac{1}{2}}\partial_qV(q_{k,h})\partial_p+\frac{1}{2}(-\Delta_p+p^2)\;,
\end{align}
and 
\begin{align}
w_{k,j}(q,p)=\frac{1}{h^{\frac{d}{4(r-1)}}}\widetilde{w}(\frac{q}{h^{\frac{1}{2(r-1)}}},p)\;.
\end{align}
Now, in order to absorb the errors in (\ref{a.3336}) we need the following estimates showed in \cite{BNV} (see (\ref{1.5mm})), 
\begin{align}
\|\widetilde{K}  _{j,V^1_{k,h}}\widetilde{w}_{k,j}\|^2_{L^2}\ge c\| (h^{-\frac{1}{2}}|\partial_qV(q_{k,h})|)^{\frac{2}{3}}\widetilde{w}_{k,j}\|^2_{L^2}\;.\label{a.32.30}
\end{align}
From now on assume $j\in\mathbb{N}.$ Taking into account (\ref{a.337}) and (\ref{a.32.30}),
\begin{align}
\|K_{j,V^1_{k,h}}\widetilde{w}_{k,j}\|^2_{L^2}\ge c\| (h^{-\frac{1}{2}})^{\frac{2}{3}}\widetilde{w}_{k,j}\|^2_{L^2}\;.\label{a.3341}
\end{align}
Owing to (\ref{a.32.666}) and (\ref{a.32.30}),
\begin{align}
\|K_{j,V^1_{k,h}}w_{k,j}\|^2\ge c\| (h^{-\frac{1}{2}})^{\frac{2}{3}}w_{k,j}\|^2\;.\label{a.33.311}
\end{align}
Note that one has
Therefore, combining (\ref{a.3336}) and (\ref{a.33.311}), there is a constant $c>0$ so that 
\begin{align}
\|K_{j,V}v_j\|^2\ge c\sum\limits_{k\ge-1}\|K_{j,V^1_{k,h}}w_{k,j}\|^2\;.\label{a.33.3222}
\end{align}
Using once more \cite{BNV} (see (\ref{eq44})), there is a constant $c>0$ such that
\begin{align}
\|\widetilde{K}_{j,V^1_{k,h}}\widetilde{w}_{k,j}\|^2\ge c\Big(\|O_p\widetilde{w}_{k,j}\|^2+\|\langle D_q \rangle^{\frac{2}{3}} \widetilde{w}_{k,j}\|^2+\|\langle h^{-\frac{1}{2}}|\partial_qV(q_{k,h})|\rangle ^{\frac{2}{3}}\widetilde{w}_{k,j}\|^2\Big)\;.\label{a.33.344}
\end{align}
As a consequence of (\ref{a.32.666}) and (\ref{a.33.344}),
\begin{align}
\|K_{j,V^1_{k,h}}w_{k,j}\|^2\ge c\Big(\|O_pw_{k,j}\|^2+\|\langle h^{\frac{1}{2(r-1)}} D_q \rangle^{\frac{2}{3}} w_{k,j}\|^2+\|\langle h^{-\frac{1}{2}}|\partial_qV(q_{k,h})|\rangle ^{\frac{2}{3}}w_{k,j}\|^2\Big)\;.\label{a.3346}
\end{align}
By (\ref{a.337}) and (\ref{a.32.24}),
\begin{align}
 \frac{1}{2}|\partial_qV(q)|\le|\partial_qV(q_{k,h})|\le \frac{3}{2}|\partial_qV(q)|\;,\label{a.3326}
\end{align}
holds for all $q\in\mathrm{supp}\;w_{k,j}$ and any $j$ large.
Then, it follows from (\ref{a.3326}) and (\ref{a.3346}),
\begin{align}
\|K_{j,V^1_{k,h}}w_{k,j}\|^2\ge c\Big(\|O_pw_{k,j}\|^2+\|\langle h^{\frac{1}{2(r-1)}} D_q \rangle^{\frac{2}{3}} w_{k,j}\|^2+\|\langle h^{-\frac{1}{2}}|\partial_qV(q)|\rangle ^{\frac{2}{3}}w_{k,j}\|^2\Big)\;.
\end{align}
Or in this case,  in vue of the (\ref{a.337}), one has $|\partial_qV(q)|\ge \epsilon_2$ for all $q\in \mathrm{supp}\; w_{k,j}.$ Hence it results from the above inequality
\begin{align}
\|K_{j,V^1_{k,h}}w_{k,j}\|^2\ge c\Big(\|O_pw_{k,j}\|^2&+\|\langle h^{\frac{1}{2(r-1)}} D_q \rangle^{\frac{2}{3}} w_{k,j}\|^2+\|(h^{-\frac{1}{2}})^{\frac{2}{3}}w_{k,j}\|^2+\|\langle h^{-\frac{1}{2}}|\partial_qV(q)|\rangle ^{\frac{2}{3}}w_{k,j}\|^2\Big)\;.\label{a.3335}
\end{align}
Furthermore, by continuity of $q\mapsto|\mathrm{Hess}\;V(q)|$ on the compact set $\overline{\mathcal{C}},$ one has for all \\$q\in\mathrm{supp}\; w_{k,j}$ and any $j$ large
\begin{align}
\frac{1}{4}(h^{-\frac{1}{2}})^{\frac{4}{3 }}\ge c\, H|\mathrm{Hess}\;V(q)|\;.
\end{align}
Then by the above inequality and (\ref{a.3335}), we get
\begin{align}
\|K_{j,V^1_{k,h}}w_{k,j}\|^2\ge c\Big(\|O_pw_{k,j}\|^2&+\|\langle h^{\frac{1}{2(r-1)}} D_q \rangle^{\frac{2}{3}} w_{k,j}\|^2+\|(2+H)^{\frac{1}{2}}w_{k,j}\|^2\nonumber\\&+\|\langle H|\mathrm{Hess}\;V(q)|\rangle^{\frac{1}{2}}w_{k,j}\|^2+\|\langle h^{-\frac{1}{2}}|\partial_qV(q)|\rangle^{\frac{2}{3}}w_{k,j}\|^2\Big)\;,\label{a.33.7.7}
\end{align}
for every $j\ge j_1$ for some $j_1\ge 1$ large. Now set
\begin{multline*}
   c_V^{(3)}=\max_{-1\le j\le j_1}
\left[\sup_{q\in \mathrm{supp}\,(\chi_{j}\theta_{k,h})}\left(\langle
H|\mathrm{Hess}~V(q)|\rangle+\langle
h^{-\frac{1}{2}}\left|\partial_{q}V(q)\right|\rangle^{4/3}\right)
\right.
\\
\left.
+\frac{(2+H)}{\log(2+H)^2}+\frac{3}{2}\,c\,\frac{(|\ln(h)|^2h^{2\nu})^2}{h}+2|\ln(h)|^{-2}h^{\frac{1}{r-1}-2\nu}\right]\,.
\end{multline*}
Seeing \eqref{eq44}, we deduce the existence of a constant $c>0$ so that 
\begin{align}
\frac{9}{64}\|K_{V_{k,h}^{1}}w_{k,j}\|&+(c_V^{(3)}-\frac{3}{2}\,c\,\frac{(|\ln(h)|^2h^{2\nu})^2}{h}-2|\ln(h)|^{-2}h^{\frac{1}{r-1}-2\nu})\|w_{k,j}\|^2\nonumber\\&\geq c(\|O_pw_{k,j}\|^{2}+\|\langle
 h^{\frac{1}{2(r-1)}} D_{q}\rangle^{2/3}w_{k,j}\|^{2} 
+\|\langle h^{-\frac{1}{2}}|\partial_{q}V(q)|\rangle^{2/3}w_{k,j}\|^{2}
\nonumber\\&\hspace{4cm}+\|\langle 
H|\mathrm{Hess}~V(q)|\rangle^{1/2}w_{k,j}\|^{2}
+\|\frac{(2+H)^{\frac{1}{2}}}{\log(2+H)}w_{k,j}\|^{2})\;,\label{a.333334}
\end{align}
holds for all $-1\le j\le j_1.$

Thus, combining the estimates (\ref{a.33.3222}), (\ref{a.33.7.7}) and (\ref{a.333334})
\begin{align}
\|K_{j,V}v_j\|^2+c_V^{(4)}\|v_j\|^2\ge c\sum\limits_{k\in I(\epsilon_1)}\Big(\|&O_pw_{k,j}\|^2+\|\langle h^{\frac{1}{2(r-1)}} D_q \rangle^{\frac{2}{3}}w_{k,j}\|^2+\|\frac{(2+H)^{\frac{1}{2}}}{\log(2+H)}w_{k,j}\|^{2}\nonumber\\&+\|\langle H|\mathrm{Hess}\;V(q)|\rangle ^{\frac{1}{2}}w_{k,j}\|^2+\|\langle h^{-\frac{1}{2}}|\partial_qV(q)|\rangle ^{\frac{2}{3}}w_{k,j}\|^2\Big)\;,\label{a.33.200}
\end{align}
holds for all $j\ge-1.$

In conclusion, in view of (\ref{a.33.2000}) and (\ref{a.33.200}), there is a constant $c>0$ such that
\begin{align}
\|K_{j,V}v_j\|^2+c_V^{(5)}\|v_j\|^2\ge c\sum\limits_{k\ge-1}\Big(\|&\frac{O_p}{\log(2+H)}w_{k,j}\|^2+\|\frac{\langle h^{\frac{1}{2(r-1)}} D_q \rangle^{\frac{2}{3}}}{\log(2+H)}w_{k,j}\|^2+\|\frac{(2+H)^{\frac{1}{2}}}{\log(2+H)}w_{k,j}\|^2\nonumber\\&+\|\frac{\langle H|\mathrm{Hess}\;V(q)|\rangle ^{\frac{1}{2}}}{\log(2+H)}w_{k,j}\|^2+\|\frac{\langle h^{-\frac{1}{2}}|\partial_qV(q)|\rangle ^{\frac{2}{3}}}{\log(2+H)}w_{k,j}\|^2\Big)\;,\label{a.3333555}
\end{align}
holds for all $j\ge -1.$ 

Finally setting $L(s) =\frac{ s+1}{\log(s+1)}$ for all $s \ge 1,$ notice that there is a constant $c > 0$ such that for all $x \ge 1,$
\begin{align}
\inf_{t\ge2}\frac{x}{\log(t)}+ t \ge \frac{ 1}{c}L(x)\;.\label{a.33555}
\end{align}
After setting the quantities 
\begin{align*}
\Lambda_{1,j}=\frac{O_p}{\log(2+H)}~,&\quad\Lambda_{2,j}=\frac{\langle H|\mathrm{Hess}\;V(q)|\rangle^{1/2}}{\log(2+H)}~,\quad\Lambda_{3,j}=\frac{\langle h^{-\frac{1}{2}}|\partial_q V(q)|\rangle^{\frac23}}{\log(2+H)}~,\\&\quad\Lambda_{4,j}=\frac{2+H}{\log(2+H)^2}\;,\quad\Lambda_{5,j}=\frac{\langle h^{\frac{1}{2(r-1)}}D_q \rangle)^{\frac{2}{3}}}{\log(2+H)}~,
\end{align*}
we get through the estimate (\ref{a.33555}), for every $j,k\ge -1$ 
\begin{align*}
\|\Lambda_{1,j}w_{k,j}\|^2_{L^2}+\frac{1}{4}\|\Lambda_{4,j}w_{k,j}\|^2_{L^2}\ge c_1\|L(O_p)w_{k,j}\|^2_{L^2}\;,
\end{align*}
\begin{align*}
\|\Lambda_{5,j}w_{k,j}\|^2_{L^2}+\frac{1}{4}\|\Lambda_{4,j}w_{k,j}\|^2_{L^2}\ge c_2\|L(\langle h^{\frac{1}{2(r-1)}} D_q \rangle^{\frac{2}{3}})w_{k,j}\|^2_{L^2}\;,
\end{align*}
\begin{align*}
\|\Lambda_{2,j}w_{k,j}\|^2_{L^2}+\frac{1}{4}\|\Lambda_{4,j}w_{k,j}\|^2_{L^2}\ge c_3\|L(\langle H|\mathrm{Hess}\;V(q)|\rangle^{\frac{1}{2}})w_{k,j}\|^2_{L^2}\;,
\end{align*}
\begin{align*}
\|\Lambda_{3,j}w_{k,j}\|^2+\frac{1}{4}\|\Lambda_{4,j}w_{k,j}\|^2_{L^2}\ge c_4\|L(\langle h^{-\frac{1}{2}}|\partial_qV(q)|\rangle^{\frac{2}{3}})w_{k,j}\|^2_{L^2}\;.
\end{align*}
From the above estimates and (\ref{a.3333555}),
\begin{align}
\|K_{j,V}v_j\|^2+c_V^{(6)}\|v_j\|^2\ge\; c\sum\limits_{k\ge-1}&\Big(\|L(O_p)w_{k,j}\|^2+\|L(\langle h^{\frac{1}{2(r-1)}} D_q \rangle^{\frac{2}{3}})w_{k,j}\|^2\nonumber\\&+\|L(\langle H|\mathrm{Hess}\;V(q)|\rangle^{\frac{1}{2}})w_{k,j}\|^2+\|L(\langle h^{-\frac{1}{2}}|\partial_qV(q)|\rangle^{\frac{2}{3}})w_{k,j}\|^2\Big)\;.
\end{align}
Therefore in view of Lemma 2.5 in \cite{Ben} conjugated by the unitary transformation of the change of scale,
\begin{align}
\|K_{j,V}v_j\|^2+c_V^{(7)}\|v_j\|^2\ge\; c\Big(\|L(O_p&)v_j\|^2+\|L(\langle h^{\frac{1}{2(r-1)}} D_q \rangle^{\frac{2}{3}})v_j\|^2\nonumber\\&+\|L(\langle H|\mathrm{Hess}\;V(q)|\rangle^{\frac{1}{2}} )v_j\|^2+\|L(\langle h^{-\frac{1}{2}}|\partial_qV(q)|\rangle^{\frac{2}{3}} )v_j\|^2\Big)\;,
\end{align}
or equivalently
\begin{align}
\|K_{V}u_j\|^2+c_V^{(7)}\|u_j\|^2\ge c\Big(\|L(O_p)&u_j\|^2+\|L(\langle D_q \rangle^{\frac{2}{3}})u_j\|^2\nonumber\\&+\|L(\langle \mathrm{Hess}\;V(q)\rangle^{\frac{1}{2}})u_j\|^2+\|L(\langle \partial_qV(q)\rangle^{\frac{2}{3}})u_j\|^2\Big)\;,
\end{align}
for every $j\ge -1.$

Therefore, combining the last estimate and (\ref{a.33111}), there is a constant $C_V>1$ so that
\begin{align}
\|K_{V}u\|^2_{L^2(\mathbb{R}^{2d})}+C_V\|u\|^2_{L^2(\mathbb{R}^{2d})}\ge \frac{1}{C_V}\Big(\|&L(O_p)u\|^2+\|L(\langle D_q \rangle^{\frac{2}{3}})u\|^2\nonumber\\&+\|L(\langle\mathrm{Hess}\;V(q)\rangle^{\frac{1}{2}})u\|^2+\|L(\langle \partial_qV(q)\rangle^{\frac{2}{3}})u\|^2\Big)
\end{align}
holds for all $u\in\mathcal{C}_0^{\infty}(\mathbb{R}^{2d}).$
\end{proof}
\textbf{Acknowledgement} I would like to thank my supervisor Francis Nier for his support and guidance throughout this
work.

\end{document}